\newcommand{\mytitle}{Linking discrete and continuum diffusion models:
Well-posedness and stable finite element discretizations}

\newcommand{\myabstract}{%
  In the context of mathematical modeling, it is sometimes convenient to integrate models of different nature. These types of combinations, however, might entail difficulties even when individual models are well-understood, particularly in relation to the well-posedness of the ensemble. In this article, we focus on combining two classes of dissimilar diffusive models: the first one defined over a continuum and the second one based on discrete equations that connect average values of the solution over disjoint subdomains. For stationary problems, we show unconditional stability of the linked problems and then the stability and convergence of its discretized counterpart when mixed finite elements are used to approximate the model on the continuum. The theoretical results are highlighted with numerical examples illustrating the effects of linking diffusive models. As a side result, we show that the methods introduced in this article can be used to infer the solution of diffusive problems with incomplete data.}

\newcommand{\myack}{C.S. and I.R. acknowledge the financial support from Madrid's regional government through a grant with IMDEA Materials addressing research activities on SARS-COV 2 and COVID-19, and financed with REACT-EU resources from the European regional development fund.}

\newcommand{\mypackages}{%
	\usepackage{amssymb}
	\usepackage{amsmath}
	\usepackage{amsthm}
	\usepackage{graphicx}
	\usepackage{xfrac}
	\usepackage{enumitem}
	\usepackage{xcolor}
}

\newcommand{\mymacros}{%
	\newcommand{\concept}[1]{\textbf{\emph{##1}}}
	\newcommand{\defined}{:=}
	\newcommand{\dV}{\;\mathrm{d}V}
	\newcommand{\dx}{\;\mathrm{d}x}
	\renewcommand{\div}{{\mathop{\mathrm{div}}}}
	\newcommand{\mbs}[1]{\boldsymbol{##1}}
	\newcommand{\pairing}[2]{\langle{##1},{##2}\rangle}
	\newcommand{\dd}[2]{\frac{\mathrm{d} ##1}{\mathrm{d} ##2}}
	\newcommand{\pd}[2]{\frac{\partial{##1}}{\partial{##2}}}
	\newcommand{\trace}{{\mathop{\mathrm{tr}}}}
	\newcommand{\uptohere}{\centerline{\textcolor{blue}{\rule{6cm}{0.2cm}}}}
	\let\oldLambda=\Lambda\renewcommand{\Lambda}{\mathit{\oldLambda}}
	\let\oldGamma=\Gamma\renewcommand{\Gamma}{\mathit{\oldGamma}}
	\newcommand{\ignacio}[1]{\par\noindent\texttt{\color{orange}$\triangleleft$~##1}}
	\newcommand{\christina}[1]{\par\noindent\texttt{\color{red}$\triangleleft$~##1}}
	\newcommand{\david}[1]{\par\noindent\texttt{\color{green}$\triangleleft$~##1}}
	\newcommand{\resolved}[1]{\par\noindent\texttt{\color{blue}##1}}
	
	\usepackage{hyperref}
	\usepackage[capitalise]{cleveref}
	\usepackage{multirow}
	\usepackage{multicol}
	\usepackage{caption}
	\usepackage{subcaption}
	\captionsetup[sub]{font=normalsize,labelfont={bf,sf}}
}

\documentclass[10pt,a4paper,preprint]{article}
\mypackages%
\mymacros%
\newcommand{\mybibstyle}{unsrt}

\newtheorem{theorem} {Theorem}[section]

\theoremstyle{definition}

\newtheorem{examplex}[theorem]{$\triangleright\;$Ejemplo}

\theoremstyle{remark}

\begin{document}


\newtheorem{remarks}{Remarks}%

\title{\mytitle}
\author{Christina Schenk$^{1}$, David Portillo$^{2,1}$, Ignacio Romero$^{2,1}$\thanks{Corresponding author: ignacio.romero@upm.es}}
\date{$^1$IMDEA Materials Institute, Eric Kandel 2, Tecnogetafe, 28906 Madrid, Spain
	\\[2ex]$^2$Universidad Polit\'ecnica de Madrid,
	Jos\'{e} Guti\'{e}rrez Abascal, 2, 28006 Madrid, Spain}
\newenvironment{acknowledgements}{\section*{Acknowledgements}}{}
	\maketitle


\begin{abstract}
	\myabstract
\end{abstract}
	

\section{Introduction}
\label{sec-intro}
Diffusion problems of interest in Applied Math and Engineering can be studied with discrete models as well as partial differential equations (PDEs). The first approach is naturally simpler than the second one since it sidesteps the difficulties that result from the spatial description of the solution fields. In addition, the solution of these simple models can be approximated very efficiently at the expense of all spatial details, in contrast with the approximation of PDEs. The latter, in fact, invariably requires working with (large) systems of equations that arise from the spatial approximation of the problem. Not surprisingly, the co-existence of a \emph{hierarchy} of models for a single physical phenomenon is a common trait of most, if not all, scientific endeavors.

Precisely because several models of different complexity often exist for one single physical problem, sometimes it proves convenient to combine several of them to exploit their relative advantages. For example, when studying complex deformable bodies, it has been proven effective to combine models for beams and solids \cite{surana1980uy,romero2018iu,steinbrecher2021ge}, since the economy of the beam equations can be exploited to analyze whole structures, whereas the (more complex) equations of solid mechanics are used to describe with detail the mechanics of regions with intricate stress distributions. This same strategy can be found in the analysis of complex --- typically multiscale --- problems in, e.g., the study of ground and subsurface hydraulic flow \cite{furman2008al}, arteries and the heart \cite{moghadam2013hs}, capillary network and major blood vessels \cite{angelo2008al,quarteroni2016xl}, etc. In these situations, always motivated by a reduction of complexity or computational cost, each of the connected models might be well-known, but linking them poses difficulties. In particular, the well-posedness of the joint problem is a delicate matter: even when each individual model is described with well-posed equations, one still needs to prove that the connection does not spoil this property. Hence, links between models of different nature are of theoretical as well as practical interest.

In this article, we analyze the coupled solution of two diffusive models, the first one defined over a continuum --- and thus described by a PDE --- and the second one defined over discrete network elements --- and described with an ordinary differential equation (ODE). A prototypical example of the problem of interest in this article is thermal equilibrium. Its most general description in a three-dimensional body employs Poisson's equation, the paradigmatic elliptic model. In addition, when a body is slender, its thermal equilibrium might be described by a second-order ODE. In practical applications, however, we might be interested in modeling the temperature of a part that is best described as a conductive body with a \emph{wire} that connects some regions, a wire that need not be inside the body. While one could model the ensemble as a continuum, it proves more convenient --- especially when using numerical discretizations --- to use different models for the bulk and the wire. This is a relevant problem, for example, in the thermal design of printed circuit boards (PCBs) where the thermal conductivity of the electronic components and their (thin) connections are very different, as also their geometries. A second relevant example that involves coupled models of different nature appears when modeling the diffusion of infectious diseases. While multi-compartmental network models have been used for centuries \cite{bernoulli1760} and many results have been obtained (e.g., \cite{martcheva2015}), they lack spatial resolution. More recent efforts --- especially related to COVID-19 --- attempt to use diffusive boundary value problems instead \cite{grave2021, GRAVE2022115541, Viguerie2020, guglielmi2022}.

Given its practical and theoretical interests, the numerical approximation of problems that mix models of different dimensionality has been studied before (e.g., \cite{laurino2019ps, steinbrecher2021ge, kuchta2021on, hagmeyer2022ac, angelo2008al}). In particular, some works have studied the approximation of mixed-dimensional problems where the low-dimension model is embedded within the high-dimensional one, with coupling \emph{fluxes} between them \cite{laurino2019ps,angelo2008al}. One of the main difficulties for coupling 3D-1D models is that the trace operator from the 3D domain to the 1D domain is not well-posed if the lower dimensional problem is a 1D manifold with a co-dimension larger than one. Here, however, we are only interested in coupling elements of a discrete (network) model with a continuum model. In contrast with the references before, the network is just a mathematical abstraction, where edges represent connections (i.e., carriers of flux) and are not, strictly speaking, 1D submanifolds of the continuum. More specifically, the networks that we will study in this work model diffusive phenomena between disjoint regions of the continuum domain, connecting the average values of the linked variables through a discrete diffusive law. Note that our solution approach differs from the so-called network diffusion models aiming to solve PDEs on discrete graphs\cite{newman2010, kuhl2021}.

Given the difficulties in formulating and analyzing general coupled continuum-discrete models, we restrict our presentation to the study of the problems that link two dissimilar models, each of them being the solution of a minimization problem. The motivation for this choice is double: first, many diffusive problems of interest have this form and, second, they are easily amenable to analysis. In particular, if two different models are of this type, their link might be conveniently tackled using the classical method of Lagrange multipliers that results in a \emph{saddle point} problem. The mathematical aspects of saddle point problems and their approximations with Galerkin-type methods~\cite{roberts1989vm,brezzi1991tn,auricchio2005tt} are well-understood both in $\mathbb{R}^n$ \cite{rockafellar1970ww} as well as in Hilbert spaces. The current interest remains, thus, in formulating new links for different models and proving that the resulting coupled problem as well as its discretization are stable.

In the current work, we study linked formulations of bulk and one-dimensional diffusive network models, aiming to provide a rigorous footing for the family of continuum/discrete problems mentioned above. The results presented herein address first the continuum model, i.~e., the well-posedness of the \emph{mixed} problem that appears when the bulk and one-dimensional discrete type problems are linked. Once this problem is studied using standard tools, we formulate mixed finite elements for discrete/continuum diffusion problems and analyze also their well-posedness. In contrast to elliptic problems, finite element discretizations of mixed problems do not inherit the well-posedness from the continuum counterpart and a different stability analysis has to be performed. In this work, we show the unconditional stability of the continuum/discrete problems of interest by proving a discrete \emph{inf-sup} condition. These theoretical findings are then illustrated with numerical examples. The main result of this work is thus that linked discrete/continuum finite element formulations of the coupled diffusive problems considered are unconditionally stable and convergent. Also, as we will show, the ideas of the proposed coupling can be used for the solution of diffusive problems with missing data or partial information. An elegant solution to this seemingly unrelated problem --- closer to data science than diffusion --- can be obtained in a straightforward manner employing the formulation presented in this work.

The article is structured as follows. In \Cref{Sec:Problem} we describe the mathematical formulation of the continuum/discrete problems. The theoretical investigations that prove the well-posedness of the joint problem are presented in \Cref{Sec:Analysis}. Then, in \Cref{Sec:ApproxProblem}, we introduce the finite element discretization of the coupled problem and prove that the well-posedness of the original  problem is inherited by the fully discrete problem. In \Cref{Sec:Results}, we demonstrate the applicability of the previously introduced concepts and methods for several problems that are of practical interest. Finally, \Cref{Sec:Conclusion} summarizes the main results of the article.

\section{Mixed-dimensional Poisson problems}\label{Sec:Problem}
In this section, we define the diffusive problems whose analysis and approximated solution are the central topic in this article. The first part of this boundary-value problem describes the stationary solution of a transport problem in a continuum and is ubiquitous in Mathematical Physics. For concreteness, in the following, we use the language of thermal analysis assuming that the unknown field is the temperature in a body. Throughout the article, all the equations could be reinterpreted in terms, for example, of mass concentration or electrical charge. The second type of problem refers to the temperature distribution on a one-dimensional \emph{wire} whose solution is given by a second-order \emph{ordinary} differential equation. In certain situations, moreover, a closed-form solution to this problem can be found, yielding a discrete diffusion relation between the temperature at the ends of the wire. Finally, we describe a joint solution of these two problems when they are solved simultaneously, that is when we consider a body in thermal equilibrium with two or more disjoint subsets connected by a thermal wire.

\subsection{The Poisson problem in a continuum domain}
\label{sec-statement}
The continuum body where we would like to study its temperature distribution is a bounded open
set $\Omega\subset\mathbb{R}^d$ with boundary denoted as $\partial\Omega$. For simplicity, in what
follows, we will restrict to $d=2$, but no fundamental problem arises in the three-dimensional case.
The temperature in this body is a field $\phi\in H_0^1(\Omega)$, the Hilbert space of (Lebesgue) square-integrable functions with square-integrable weak first derivatives and vanishing trace at the boundary. If $f\in H^{-1}(\Omega)$ is a known field of heat supply, the stationary Poisson problem can be written in its standard strong form
\begin{equation}
	\begin{split}
		-\kappa\, \Delta \phi &= f \qquad  \text{ in } \Omega\ ,\\
		\phi &= 0 \qquad \text{ on } \partial\Omega\ .\\
	\end{split}
	\label{eq-poisson}
\end{equation}
Here $\kappa$ is the thermal conductivity of the medium and will be assumed to be constant, for simplicity, and $\Delta$ is the Laplacian operator. See,
e.g., \cite{evans1999wj} for a detailed description of this canonical elliptic boundary-value problem.

With a view to the analysis and discretization of Eq.~\eqref{eq-poisson}, we rewrite the Poisson problem in weak form. For
that we define $\mathcal{U}\equiv H_0^1(\Omega)$ and recall that its norm is given, for any $\phi\in\mathcal{U}$ by the standard expression
\begin{equation}
\label{eq-spaceU}
	\| \phi\|_{\mathcal{U}} :=
	\left(\|\phi\|^2_{L^2(\Omega)}+\ell^2\|\nabla \phi\|^2_{L^2(\Omega) }\right)^{1/2} ,
\end{equation}
where $\ell$ is the characteristic length of the problem (for example, the diameter of $\Omega$)
and $\nabla$ denotes the gradient operator. Then, the weak form of Poisson's problem consists in finding $\phi\in\mathcal{U}$ such that
\begin{equation}
  \label{eq-poisson-weak}
  a_{\Omega}(\phi, \psi)
  =
  f_{\Omega}(\psi)
\end{equation}
for all $\psi \in\mathcal{U}$, where
\begin{equation}
	a_{\Omega} (\phi, \psi) := \int_{\Omega} \nabla \phi \cdot \kappa \nabla \psi \;\mathrm{d}V ,
	\qquad
	f_{\Omega}(\psi) := \int_{\Omega} f \cdot \psi \; \mathrm{d}A,
\end{equation}
are, respectively, a bilinear and a linear form on $\mathcal{U}$.

\subsection{The Poisson problem on a segment}
We describe next the stationary heat problem for a one-dimensional body, which is referred to throughout as a \emph{wire}. The governing equations of this problem can be derived from the statement~\eqref{eq-poisson} of the Poisson problem, simply by assuming that the body $\Omega$ has a prismatic shape and the temperature field is constant in all the points of a cross section. Details of this projection are omitted and the final form of the equation is given.

Consider a wire of length $L$ with temperature $\theta:[0,L]\to \mathbb{R}$. When the wire is in thermal equilibrium, the temperature must
satisfy
\begin{equation}
    \label{eq-wire-strong}
	\bar{\kappa} \dfrac{\mathrm{d}^2\theta}{\mathrm{d}x^2} = r ,
\end{equation}
where $r$ is the heat supply per unit length and $\bar{\kappa}$ is the cross-sectional conductivity. The problem, however, is not well-posed unless we append suitable boundary conditions. Rather, the temperature is only defined modulo an affine function.

As before, we are interested in the weak formulation of this problem. Postponing for the moment the issue of the uniqueness in the solution, we can introduce the function space $\mathcal{W}=H^1(0,L)$, analogous to the solution space for the body but without the trace constraint, and look for solutions $\theta\in\mathcal{W}$ such that
\begin{equation}
\label{eq-wire-weak}
 a_{L} (\theta, \eta) = f_{L}(\eta),
\end{equation}
for all $\eta\in \mathcal{W}$, where
\begin{equation}
    a_{L} (\theta, \eta) :=
    \int_0^L\bar{\kappa}\;\theta'\eta'\; \mathrm{d}x ,
    \qquad
    f_{L}(\eta) :=
    \int_0^L r \; \eta\; \mathrm{d}x .
\end{equation}
For future reference, we recall that for every function $\theta\in\mathcal{W}$, its norm is
\begin{equation}
	\| \theta\|_{\mathcal{W}} := \left(\|\theta\|^2_{L^2(0,L)}+L^2\,\|\theta'\|^2_{L^2(0,L)) }\right)^{1/2} .
\end{equation}

Let us note that the actual geometry of the wire plays no role in the equations above, except for its length. As explained in Sec.~\ref{sec-intro}, Eq.~\eqref{eq-wire-weak} merely defines the behavior of a network edge, an abstract entity that connects temperature of two points through a diffusive equation. Moreover, if the heat supply $r$ is identically zero and the diffusivity is constant, the solution to Eq.~\eqref{eq-wire-strong} is an affine function. In this case, for all practical purposes, the wire will just establish a \emph{discrete} diffusive relation between the temperature at its two end points.

\subsection{Linked formulation}\label{sec:linkedform}

We study next a joint problem consisting of a body $\Omega$ with a temperature field~$\phi$ satisfying the problem~\eqref{eq-poisson-weak}, where additionally we identify two disjoint non-empty regions $B_1,B_2\subsetneq\Omega$ whose temperatures are connected by means of a thermally conductive wire. This wire is such that the temperature at each of its ends coincides with the \emph{mean} temperature of the regions $B_1,B_2$, respectively. See \cref{Fig:sketch} for an illustration
of the linked problem. We note that, alternatively, we could have considered two disjoint conductive bodies $\Omega_1$ and $\Omega_2$ that are in thermal equilibrium while a wire connects regions $B_\alpha\subset\Omega_\alpha$, with $\alpha=1,2$. The analysis of both of the problems described is analogous and, for conciseness, we focus on the former.

\begin{figure}[htb!]
	\centering
	\includegraphics[scale=.45]{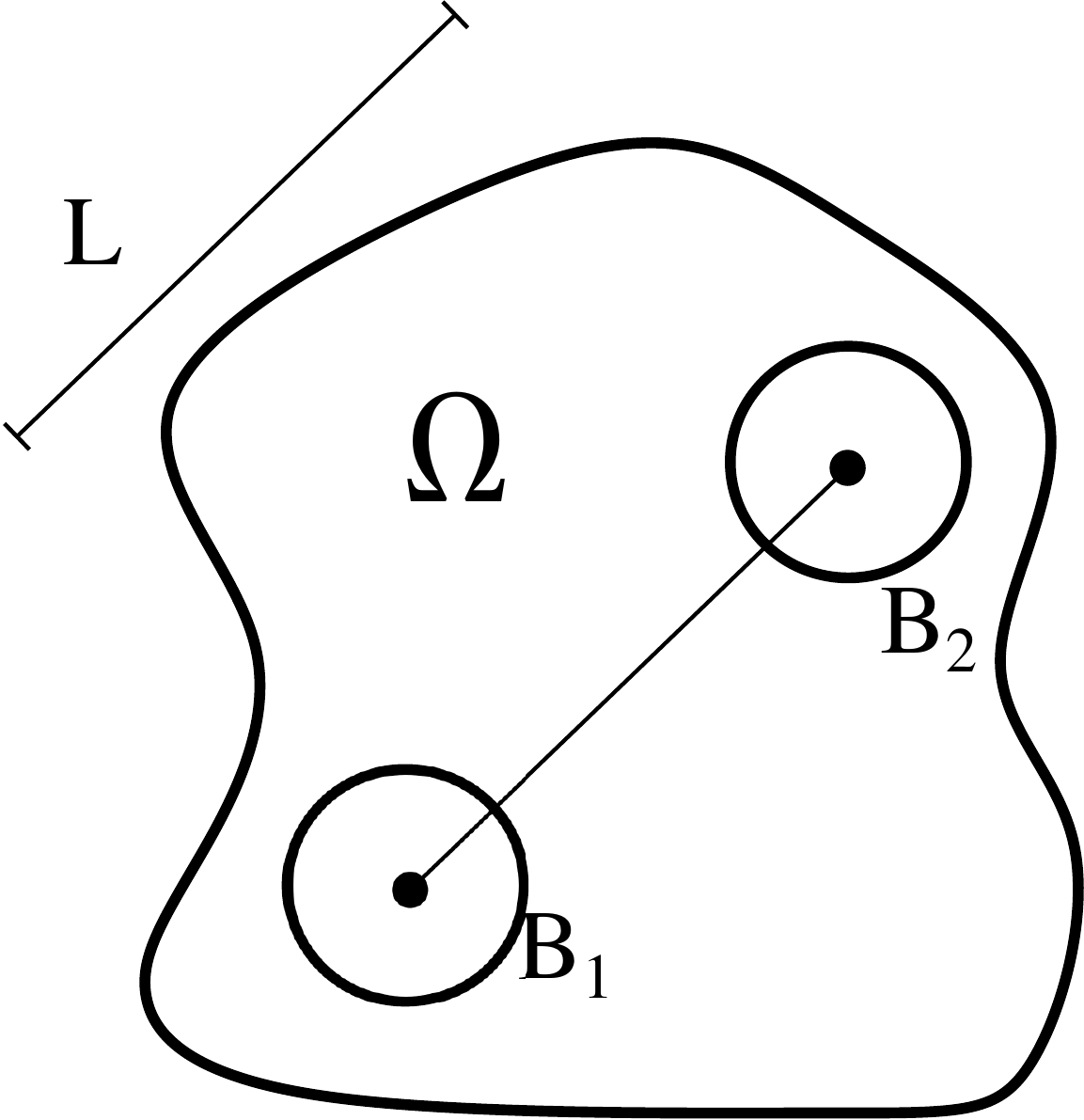}
	\caption{Sketch of linked solid-wire problem setting.}
	\label{Fig:sketch}
\end{figure}

Let us insist, once again, that the purpose of this coupled model is not the actual representation of a true wire that connects two points in the body (say, the centers of $B_1$ and $B_2$ in Fig.~\ref{Fig:sketch}). Rather, the wire connects regions $B_1$ and $B_2$ through a discrete diffusive equation. More precisely, let $\theta_1,\theta_2$ denote the mean temperature in the body
regions $B_1,B_2$, that is,
\begin{equation}
    \theta_{\alpha} := \frac{1}{\vert B_\alpha\vert}\int_{B_\alpha}\phi\; \mathrm{d}V,
    \label{eq-mean-temperatures}
\end{equation}
with $\alpha=1,2$, and $\vert B_\alpha\vert$ denoting the non-zero measure of the set. Then, the problem governing the temperature field on the wire can be fully described by the boundary value problem
\begin{subequations}
\label{eq-linked-wire}
	\begin{align}
		\bar{\kappa} \dfrac{d^2\theta}{dx^2} &=0 \qquad \text{ on } (0,L), \label{eq-linked-de}\\
		\theta(0)&=\theta_1\ , \label{eq-linked-t1}    \\
		\theta(L)&=\theta_2\ .  \label{eq-linked-t2}
	\end{align}
\end{subequations}
If the scalars $\theta_1,\theta_2$ were known, the problem~\eqref{eq-linked-wire} would be standard and its well-posedness would require no further analysis. However, here we are interested in the situation where the values $\theta_1,\theta_2$ are not known \emph{a priori} but rather, obtained through the averages~\eqref{eq-mean-temperatures} of the solution to problem~\eqref{eq-poisson-weak}.

It remains, thus, to formulate the coupled problem that includes the thermal equilibrium
of the body and wire, as well as the link conditions. We will use Lagrange multipliers $(\lambda_1,\lambda_2)=:\lambda\in Q\equiv\mathbb{R}^2$ to
enforce the two constraints~\eqref{eq-linked-t1} and~\eqref{eq-linked-t2}, and we will
use the notation $\|\cdot\|_Q$ to indicate the Euclidean norm in $\mathbb{R}^2$. For
convenience, we introduce the product space $\mathcal{V}=\mathcal{U}\times\mathcal{W}$ with norm
\begin{equation}
	\| (\phi,\theta)\|_{\mathcal{V}} :=
	\left(
	\|\phi\|^2_{\mathcal{U}} +
	\ell^{d-1}\, \|\theta\|^2_{\mathcal{W}}\right)^{\frac{1}{2}} ,
\end{equation}
for all $(\phi, \theta) \in\mathcal{V}$.
On this space we can define the bilinear form $a(\cdot,\cdot):\mathcal{V}\times\mathcal{V}\to\mathbb{R}$ and linear form $f:\mathcal{V}\to\mathbb{R}$ as
\begin{equation}
\label{eq-complete-forms}
\begin{aligned}
	a (\phi, \theta; \psi, \delta) &:= a_{\Omega}(\phi,\psi)+a_L(\theta,\delta)\ ,\\
	f(\psi, \delta) &:= f_{\Omega}(\psi) +f_L(\delta)  \in\mathcal{V} .\\
	\end{aligned}
\end{equation}
for all $(\phi,\theta)$ and $(\psi,\delta)$ in $\mathcal{V}$.

The joint equilibrium of the solid and the wire then results from the saddle point of the Lagrangian $\mathcal{L}:\mathcal{V}\times Q\to \mathbb{R}$:
\begin{equation}
	\begin{split}
	\mathcal{L}(\phi,\theta,\lambda_1,\lambda_2) :=&
\frac{1}{2}a(\phi,\theta;\phi,\theta)-f(\phi,\theta) \\
&+ \lambda_1\left(\theta(0)-\frac{1}{\vert B_1 \vert} \int_{B_1} \phi\;\mathrm{d}V \right)
+ \lambda_2\left(\theta(L)-\frac{1}{\vert B_2 \vert} \int_{B_2} \phi\;\mathrm{d}V\right).
	\end{split}
\label{eq:lagrangian}
\end{equation}
Hence, we aim to solve the following problem
\begin{equation}
(\phi,\theta,\lambda_1,\lambda_2)^{\star} = \arg \inf_{\phi,\theta} \sup_{\lambda}\mathcal{L}(\phi,\theta,\lambda_1,\lambda_2)
\label{eq:saddletotal}
\end{equation}
where the Lagrangian as in \cref{eq:lagrangian}.
The optimality conditions of the functional~$\mathcal{L}$ are satisfied by the functions
$(\phi,\theta,\lambda_1,\lambda_2)\in\mathcal{V}\times Q$ such that
	\begin{align}
		a(\phi,\theta;\psi,\delta) + b(\psi,\delta;\lambda_1,\lambda_2) &= f(\psi,\delta)\ ,\\
		b(\phi,\theta;\Gamma_1,\Gamma_2) &= 0,
		\label{eq:saddle}
	\end{align}
for all $(\psi,\delta,\Gamma_1,\Gamma_2)\in \mathcal{V}\times Q$ with
\begin{equation}
   b(\phi,\theta; \Gamma_1,\Gamma_2) :=
  \Gamma_1\left(\theta(0)-\frac{1}{\vert B_1 \vert}\int_{B_1}\phi \,\;\mathrm{d}V\right)+
   \Gamma_2\left(\theta(L)-\frac{1}{\vert B_2 \vert}\int_{B_2}\phi\,\;\mathrm{d}V\right).
\end{equation}

The bilinear form $a(\cdot,\cdot)$ defines a linear continuum operator $A:\mathcal{V}\to\mathcal{V}'$ by the relation
\begin{equation}
  \langle Au,v\rangle_{\mathcal{V}'\times\mathcal{V}}=a(u,v),
\end{equation}
for all $u=(\phi,\theta)\in\mathcal{V}, v=(\psi,\delta)\in\mathcal{V}$. Also, the bilinear form $b(\cdot,\cdot)$ on $\mathcal{V}\times Q$ defines a linear operator $B:\mathcal{V}\to Q'$ with transpose $B^T:Q\to\mathcal{V}'$ by
\begin{equation}
  \langle B v, \Gamma \rangle_{Q'\times Q} = \langle v,B^T\Gamma
  \rangle_{\mathcal{V}\times\mathcal{V}'}= b(v, \Gamma),
\end{equation}
for all $v\in\mathcal{V},\Gamma\in Q$. Employing these definitions, Eq.~\eqref{eq:saddle} can be alternatively rewritten as:
\begin{equation}
	\begin{aligned}
		 Au + B^T\lambda &= f\qquad\text{ in } \mathcal{V}'\\
		 Bu  &= 0 \qquad\text{ in }Q'.
		\label{eq:saddleA}
	\end{aligned}
\end{equation}

This last expression is the standard form of a \emph{mixed problem} \cite{brezzi1991tn}.
The solvability of this problem depends on conditions over the bilinear forms $a(\cdot,\cdot)$ and $b(\cdot,\cdot)$
as well as properties of the spaces where they are defined on. In particular,
often the properties of the linear operator $B$ are delicate to ascertain.

\begin{remarks}
Two modifications of problem~\eqref{eq:saddle} are interesting in their own right:

\begin{enumerate}
\item One could consider that the one-dimensional diffusive mechanism connects, instead
of disjoint subsets $B_1,B_2\subset\Omega$, the boundary of two different bodies, or parts of
them. The description of this modified problem and the functional setting are slightly
different than the one presented up to here, since the new problem will be formulated
in terms of the traces of functions.

\item Second, we could consider a situation where there is no connecting one-dimensional diffusive
wire between regions of the body but only that the \emph{average} temperature on a measurable set~$S\subset\Omega$
is known to have a fixed value~$\bar{\theta}$. In this simple case we will be left with Poisson's problem with a constraint. The problem will still be of the form~\eqref{eq:saddle}, more precisely,

\begin{equation}
   \label{eq-one-theta}
    \begin{aligned}
          a_\Omega(\phi,\psi) + \bar{b}(\psi,\gamma) &= f_\Omega(\psi)\ , \\
          \bar{b}(\phi,\eta) &= 0, \\
    \end{aligned}
\end{equation}
with $\bar{b}(\cdot, \cdot)$ on $\mathcal{V}\times\mathbb{R}$ defined as
\begin{equation}
   \label{eq-one-theta2}
    \bar{b}(\phi,\eta) :=  \eta\left( \bar{\theta} - \frac{1}{\vert S \vert}\int_S \phi\, \mathrm{d}V \right).
\end{equation}

\end{enumerate}
\end{remarks}

\section{Analysis}\label{Sec:Analysis}
In this section, we study the well-posedness of problem~\eqref{eq:saddle}. According to the standard
theory of mixed problems \cite{roberts1989vm,brezzi1991tn,auricchio2005tt}, we need to show that the
bilinear forms $a(\cdot,\cdot)$ and $b(\cdot,\cdot)$ are continuum, that $a(\cdot,\cdot)$ is
elliptic on the kernel of the operator $B$, and that a certain \emph{inf-sup} condition, to be defined
later, also holds for~$b(\cdot,\cdot)$.

In the following, for simplicity, let us assume that the conductivities $\kappa$ and $\bar{\kappa}$ are constant and
positive. The first step of the analysis is to study the continuity of all the linear and bilinear
forms appearing in the problem statement~\eqref{eq:saddle}. This is trivial and we summarize all the
results, without proof, in the following theorem:
\begin{theorem}
	The bilinear forms $a_{\Omega}(\cdot,\cdot)$, $a_L(\cdot,\cdot)$, and $a(\cdot,\cdot)$
	are continuous in their corresponding spaces of definition, i.e.,
	\begin{equation}
	  \begin{aligned}
	  \vert a_\Omega(\phi,\psi) \vert
	  &\le c_{\Omega} \; \|\phi\|_{\mathcal{U}} \; \|\psi\|_{\mathcal{U}} \ ,
	  \\
	  \vert a_L(\theta,\delta)\vert
	  &\le c_{L} \; \|\theta\|_{\mathcal{W}} \; \|\delta\|_{\mathcal{W}}   \ ,
	 \\
	  \vert a(\phi,\theta;\psi,\delta) \vert
	  &\le c\;
	  \|(\psi,\delta)\|_{\mathcal{V}} \ ,
	  \end{aligned}
	\end{equation}
for all $\phi, \psi\in\mathcal{U}$, $\theta, \delta\in\mathcal{W}$ and some generic positive
constants $c_{\Omega}, c_L$ and $c$.

\label{Theor:aomega_cont}
\end{theorem}
Next, we show the continuity of $b(\cdot,\cdot)$. Since this bilinear form is non-standard, we provide the full proof of the result.
\begin{theorem}
 	The bilinear form $b$ is continuous on $(\mathcal{V}\times Q)$, i.e.
 	\begin{equation}
 		\vert b(\phi,\theta;\lambda) \vert \le c \; \|(\phi,\theta)\|_{\mathcal{V}} \; \|\lambda\|_{Q}
    \end{equation}
 	for all $(\phi,\theta; \lambda) \in \mathcal{V}\times Q$ and some constant $c>0$.
 	\label{Theor:al_cont}
\end{theorem}
\begin{proof}
Since $H^1(0,L)\hookrightarrow C^0[0,L]$, we can use the mean value theorem to determine
that there exists $m\in[0,L]$ such that
\begin{equation}
    \theta(m) = \frac{1}{L} \int_0^L \theta(x) \dx\ .
\end{equation}
Then, by the fundamental theorem of calculus
\begin{equation}
    \theta(L) = \theta(m) + \int_m^L \theta'(x) \dx .
\end{equation}
Hence,
\begin{equation}
\begin{aligned}
    \vert \theta(L) \vert &\le
    \frac{1}{L} \int_0^L \vert \theta(x)\vert \dx
    +
    \int_m^L \vert \theta'(x)\vert \dx
    \\
    &\le
    L^{-1/2}\, \| \theta \|_{L^2(0,L)} +
    {L}^{1/2}\, \| \theta' \|_{L^2(0,L)}
    \\
    &=
    {L}^{-1/2}\, \| \theta \|_{H^1(0,L)} .
\end{aligned}
\end{equation}
Similarly, the bound $\vert\theta(0)\vert \le {L}^{-1/2}\, \| \theta \|_{H^1(0,L)}$ also holds.
Using these results, we proceed to bound from above the bilinear form~$b(\cdot,\cdot)$:
\begin{align*}
    \vert b(\phi,\theta;\lambda)\vert
    &=
    	\left\vert\lambda_1\left( \frac{1}{\vert B_1\vert} \int_{B_1}\phi \dV -
    	\theta(0)\right)+\lambda_2\left(\frac{1}{\vert B_2 \vert}\int_{B_2}\phi \dV - \theta(L)\right) \right\vert
    	\\
    	&\leq \vert\lambda_1\vert \left\vert  \frac{1}{\vert B_1\vert}  \int_{B_1}\phi \dV-\theta(0)\right\vert +
    	\vert\lambda_2\vert \left \vert \frac{1}{\vert B_2 \vert} \int_{B_2}\phi \dV - \theta(L)\right\vert
    	\\
    	&\leq \vert \lambda_1\vert
    	\left(\frac{1}{\vert B_1 \vert} \int_{B_1} \vert\phi\vert \dV +\vert\theta(0)\vert\right)
    	+
    	\vert \lambda_2\vert  \left( \frac{1}{\vert B_2 \vert} \int_{B_2}\vert\phi\vert \dV+\vert\theta(L)\vert\right)
    	\\
    	& \leq
    	C\, \|\lambda\|_2
    	\left(
    	\frac{1}{\vert \Omega \vert} \int_\Omega \vert \phi \vert \dV +
    	L^{-1/2}\, \|\theta\|_{H^1(0,L)}
    	\right)
    	\\
    	&\leq
    	C\, \|\lambda\|_2
    	\left( \ell^{-d/2} \|\phi\|_{H_0^1(\Omega)} + L^{-1/2} \|\theta\|_{H^1(0,L)} \right)
    	\\
    	&\leq
    	C\;\ell^{-d/2}
    	\|\lambda\|_Q\; \|(\phi,\theta)\|_{\mathcal{V}},
\end{align*}
where, throughout the proof, $C$ denotes a constant whose value might change from
one step to another and $d$ denotes the dimension of the characteristic length corresponding to $\Omega$.
\end{proof}
Next, we have to ensure ellipticity on $\ker B$, the kernel of the operator~$B$.
By definition, this set is
\begin{equation}
  \ker B := \{(\phi,\theta)\in\mathcal{V}: \;b(\phi,\theta;\lambda)=0, \quad \forall \lambda \in Q\} .
\end{equation}
Elements $(\phi,\theta)\in\mathcal{V}$ in $\ker B$ verify
\begin{equation}
    \lambda_1\left(\dfrac{1}{\vert B_1\vert}\int_{B_1}\phi \;\mathrm{d}V-\theta(0)\right) +
    \lambda_2\left(\dfrac{1}{\vert B_2 \vert}\int_{B_2}\phi \;\mathrm{d}V -\theta(L)\right)
    =0.
\end{equation}
for all pairs $(\lambda_1,\lambda_2)\in Q$.
Since the two Lagrange multipliers are independent, it must hold that
\begin{equation}
    \ker B =\left\{
     (\phi,\theta)\in\mathcal{V}: \dfrac{1}{\vert B_1\vert}\int_{B_1}\phi \;\mathrm{d}V=\theta(0)\ \text{ and }\
     \dfrac{1}{\vert B_2 \vert}\int_{B_2}\phi \;\mathrm{d}V=\theta(L)
    \right\}.
\end{equation}
As advanced above, for the well-posedness of the mixed problem, it suffices that
the bilinear form $a(\cdot,\cdot)$ be elliptic on $\ker B\subset \mathcal{V}$, as shown
in the following theorem:
\begin{theorem}
	The bilinear form $a(\cdot,\cdot)$ is elliptic on $\ker B$, i.e., there exists
	a constant $\bar{\alpha}>0$ such that
	\begin{equation}
	  a(\phi,\theta; \phi,\theta)\ge \bar{\alpha}\; \|(\phi,\theta)\|^2_{\mathcal{V}} ,
	\end{equation}
	for all $(\phi,\theta)\in \ker B$.
\label{Theor:ellipt}
\end{theorem}
\begin{proof}
To start the proof, we need two preliminary results. First, using the weighted
Young inequality we note that
\begin{equation}
\label{eq-proof20}
   \begin{aligned}
    \|\theta - &\theta(0)\|^2_{L^2(0,L)} +
    \|\theta - \theta(L)\|^2_{L^2(0,L)}
    \\
    =&
    2\,\|\theta\|^2_{L^2(0,L)} + L\,\vert \theta(0) \vert^2 + L\,\vert \theta(L) \vert^2 -
    2\int_0^L \theta(0)\, \theta \dx - 2\int_0^L \theta(L)\, \theta \dx
    \\
    \ge&
    2\,\|\theta\|^2_{L^2(0,L)} + L\,\vert \theta(0) \vert^2 + L\,\vert \theta(L) \vert^2
    \\
    &
    - \epsilon_0^2\, \|\theta \|^2_{L^2(0,L)}
    - \frac{L}{\epsilon_0^2} \vert \theta(0) \vert^2
    - \epsilon_L^2\, \|\theta \|^2_{L^2(0,L)}
    - \frac{L}{\epsilon_L^2} \vert \theta(L) \vert^2
    \\
    \ge&
    \left( 2 - \epsilon_0^2 - \epsilon_L^2 \right) \|\theta\|^2_{L^2(0,L)}
    + L\left( 1 - \frac{1}{\epsilon_0^2} \right) \vert \theta(0) \vert^2
    + L\left( 1 - \frac{1}{\epsilon_L^2} \right) \vert \theta(L) \vert^2 ,
   \end{aligned}
\end{equation}
for arbitrary scalars $\epsilon_0,\epsilon_L>0$. In the kernel of $B$, moreover, we have
that
\begin{equation}
\label{eq-proof21}
    \begin{aligned}
    \vert \theta(0) \vert^2 &=
    \frac{1}{\vert B_1\vert} \left\vert \int_{B_1} \phi \dV \right\vert^2
    \le
    C\, \ell^{-d} \|\phi\|^2_{H^1_0(\Omega)} ,
    \\
    \vert \theta(L) \vert^2 &=
    \frac{1}{\vert B_2\vert} \left\vert \int_{B_2} \phi \dV \right\vert^2
    \le
    C\, \ell^{-d} \|\phi\|^2_{H^1_0(\Omega)}\ .
    \end{aligned}
\end{equation}
Assuming that the following conditions hold
\begin{equation}
    \label{eq-epsilon-C1}
    1 - \frac{1}{\epsilon_0^2} \le 0 \ ,\qquad
    1 - \frac{1}{\epsilon_L^2} \le 0 \ ,\qquad
\end{equation}
then, combining \cref{eq-proof20,,eq-proof21}, the following bound is
obtained
\begin{equation}
\label{eq-proof22}
   \begin{aligned}
    \|\theta - &\theta(0)\|^2_{L^2(0,L)} +
    \|\theta - \theta(L)\|^2_{L^2(0,L)}
    \\
    \ge&
    \left( 2 - \epsilon_0^2 - \epsilon_L^2 \right) \|\theta\|^2_{L^2(0,L)}
    + C\, \ell^{1-d} \left( 2 - \frac{1}{\epsilon_0^2} - \frac{1}{\epsilon_L^2}\right)
    \|\phi\|^2_{H^1_0(\Omega)}\ .
   \end{aligned}
\end{equation}
Also, as in the proof of \cref{Theor:al_cont}, we use the fundamental theorem
of calculus to bound
\begin{equation}
    \label{eq-proof23}
    \begin{aligned}
    \| \theta' \|^2_{L^2(0,L)} &\ge C L^{-2} \| \theta - \theta(0)\|^2_{L^2(0,L)}\ , \\
    \| \theta' \|^2_{L^2(0,L)} &\ge C L^{-2} \| \theta - \theta(L)\|^2_{L^2(0,L)}\ . \\
    \end{aligned}
\end{equation}
To finally prove the ellipticity of the bilinear form $a(\cdot,\cdot)$ we
use the definition~\eqref{eq-complete-forms}, the bounds~\eqref{eq-proof22}, \eqref{eq-proof23} and the Poincar\'e inequality from where it follows that
\begin{equation*}
	\begin{aligned}
		a(\phi,\theta;\phi,\theta) =&
		\int_{\Omega} \kappa\;\vert\nabla\phi\vert^2 \dV + \int_0^L\bar{\kappa}\;\vert\theta'\vert^2\dx
		\\
		\geq&
		C\,\kappa\,\ell^{-2} \|\phi\|^2_{H^1_0(\Omega)}
		+ \frac{\bar{\kappa}}{2} \| \theta' \|^2_{L^2(0,L)}
		+ \frac{\bar{\kappa}}{2} \| \theta' \|^2_{L^2(0,L)}
		\\
		\geq&
		C\,\kappa\,\ell^{-2} \|\phi\|^2_{H^1_0(\Omega)}
		+ C\, \kappa\, \ell^{d-1} \| \theta' \|^2_{L^2(0,L)}
		\\
		&+ C\,\kappa\,\ell^{d-3}
		\left(
		\| \theta - \theta(0)\|^2_{L^2(0,L)} + \| \theta - \theta(L)\|^2_{L^2(0,L)}
		\right)
		\\
		\geq&
		C\,\kappa\,\ell^{-2} \|\phi\|^2_{H^1_0(\Omega)}
		+ C\, \kappa\, \ell^{d-1} \| \theta' \|^2_{L^2(0,L)}
		\\
		&+ C\,\kappa\,\ell^{d-3}
		\left( 2 - \epsilon_0^2 - \epsilon_L^2 \right) \|\theta\|^2_{L^2(0,L)}
		\\
		&+
		C\, \kappa\,\ell^{-2}
		\left( 2 - \frac{1}{\epsilon_0^2} - \frac{1}{\epsilon_L^2}\right)
        \|\phi\|^2_{H^1_0(\Omega)} ,
\end{aligned}
\end{equation*}
where, as usual, $C$ denotes a generic constant whose value might change
in each inequality. Finally, we rewrite this bound as
\begin{equation*}
    \begin{aligned}
		a(\phi,\theta;\phi,\theta) \ge&
		C\,\kappa\,\ell^{-2}
		\left( 3 - \frac{1}{\epsilon_0^2} - \frac{1}{\epsilon_L^2}\right)
        \|\phi\|^2_{H^1_0(\Omega)}
        \\
        &+
        C\,\kappa\,\ell^{d-3}
        \left(
        (2-\epsilon_0^2-\epsilon_L^2)\|\theta\|^2_{L^2(0,L)}
        + L^2 \|\theta'\|^2_{L^2(0,L)}
        \right)
		\\
		\ge&
		C\,\kappa\,\ell^{-2}
        \left(
        \|\phi\|^2_{\mathcal{U}} + \ell^{d-1} \|\theta\|^2_{\mathcal{W}}
        \right)
        \\
        =&
		C\,\kappa\,\ell^{-2}
        \|(\phi,\theta)\|^2_{\mathcal{V}}\ ,
	\end{aligned}
\end{equation*}
as long as there exist $\epsilon_0^2,\epsilon_L^2 \ge 0$ that verify
conditions~\eqref{eq-epsilon-C1} as well as
\begin{equation*}
    3 - \frac{1}{\epsilon_0^2} - \frac{1}{\epsilon_L^2} > 0
    \ , \qquad
    2 - \epsilon_0^2 - \epsilon_L^2 > 0\ .
\end{equation*}
These three conditions are satisfied, for example, by $\epsilon_0^2=\epsilon_L^2=4/5$,
completing the proof.
\end{proof}

Let us note that the bilinear form $a(\cdot,\cdot)$ is not elliptic in the whole
space $\mathcal{V}$, because the bilinear form of the one-dimensional conductor,
namely $a_L(\cdot,\cdot)$ is not elliptic in the space $\mathcal{W}$, precisely
due to the lack of Dirichlet boundary conditions in problem~\eqref{eq-wire-strong}.

Finally, to ensure that the mixed problem is well-posed, we also have to show that
the following \emph{inf-sup} condition holds for $b(\cdot,\cdot)$.
\begin{theorem}\label{Theor:infsup}
Inf-sup condition. There exists a constant $\beta>0$ such that
\begin{equation}
  \sup_{(\phi,\theta)\in \mathcal{V}} \frac{b(\phi,\theta;\lambda)}{\|(\phi,\theta)\|_\mathcal{V}}
  \geq \beta
  \|\lambda\|_{Q \backslash\ker B^T}.
\end{equation}
\end{theorem}
\begin{proof}
Since $B$  is an operator from $\mathcal{V}$ to $Q'$, a finite-dimensional space, its range is closed (Theorem 1.1. in \cite{brezzi1991tn}). Hence, the inf-sup condition holds.
\end{proof}
With \cref{Theor:aomega_cont,,Theor:al_cont,,Theor:ellipt,,Theor:infsup}, it follows that problem (\ref{eq:saddle}) is well-posed. Let us remark that, typically, in other constrained minimization problems, the proof of the \emph{inf-sup} condition might be quite involved. In the problem discussed in this article, however, this proof is trivial.

\section{Approximation of the problem}\label{Sec:ApproxProblem}
To solve the saddle-point problem \eqref{eq:saddle} we use a mixed finite element discretization.
We now briefly recall the details of this method as it applies to the project at
hand.

The first step in the finite element approximation of problem \eqref{eq:saddle} is the definition of
a finite-dimensional subspace, $\mathcal{V}_h\subset V$ of the infinite-dimensional solution space.
Moreover, we assume that all the finite element functions in $\mathcal{V}_h$ are linear
combinations of piecewise polynomials defined in their corresponding domains,
namely, $\Omega$ or $[0,L]$.  Since the space of the Lagrange multipliers is already a space of
dimension~2 we do not need to introduce a subspace for it and we simply define $Q_h\equiv Q$.

The (mixed) Galerkin finite element method consists in finding
$u_h=(\phi_h,\theta_h)\in \mathcal{V}_h,\; \lambda_h\in Q_h$ such that
\begin{equation}
  \label{eq-fem}
  \begin{aligned}
    	a(u_h,v_h) +b(v_h,\lambda_h) &= f(v_h) ,\\
		b(u_h,\Gamma_h) &= 0 ,
  \end{aligned}
  \end{equation}
holds  for all $v_h=(\psi_h,\delta_h)\in \mathcal{V}_h$ and $\Gamma_h\in Q_h$. In contrast with the
finite element approximations of elliptic boundary value problems, the well-posedness of mixed
methods such as~\eqref{eq-fem} does not follow from the well-posedness of the
corresponding continuous problem. Instead, a new analysis has to be carried out and, in particular,
a discrete \emph{inf-sup} condition needs to be proven. This is typically the most difficult
ingredient of this analysis but, as we will show below, it is not the case for the problem at hand
given the finite dimension of the space~$Q$.

\subsection{Analysis}
The well-posedness of the discrete problem~\eqref{eq-fem} can be established by
following similar steps as in the analysis of the continuous
problem and presented in \cref{Sec:Analysis}. For that, we start
by introducing $B_h$, the discrete counterpart of the operator $B$ used
in \cref{Sec:Analysis}, and now defined by the relation
\begin{equation}
  \label{eq-bh}
  \langle B_{h} v_{h}, \Gamma_{h} \rangle_{Q_{h}'\times Q_{h}} =
  \langle v_{h} ,B^T_{h}\Gamma_h \rangle_{\mathcal{V}\times\mathcal{V}'}=
  b(v_{h}, \Gamma_{h}),
\end{equation}
for all $v_{h}\in\mathcal{V}_{h}$ and $\Gamma_{h}\in Q_{h}$.
Since $B_h$ is surjective and $\ker \mathcal{B}_h^T\subset \ker \mathcal{B}^T$,  $\ker B_h^T = \{0\}$ and $B_h = B{\big|}_{V_h}$. So, we end up in the special case of $\ker B_h\subset \ker B$. Hence, the ellipticity of $a(\cdot,\cdot)$ on $\ker B_h$ follows from the ellipticity of $a(\cdot,\cdot)$ on $\ker B$.
As advanced, the key
condition for the well-posedness of \eqref{eq-fem} is thus the \emph{discrete inf-sup}
condition. However, given that the range of $B_h$ is finite dimensional, it is
closed and thus the \emph{discrete inf-sup} condition holds. From Theorem 2.1. in
\cite{brezzi1991tn} we obtain the following estimate
\begin{equation}
  \| u-u_h\|_{\mathcal{V}} + \|\lambda-\lambda_h\|_{Q\backslash\ker{B^T}} \le
  c \left(
  \inf_{v_h\in \mathcal{V}_h} \|u-v_h\|_{\mathcal{V}}+
  \inf_{q_h\in Q_h}\|\lambda-\Gamma_h\|_Q\right),
\end{equation}
where $c$ is a constant that depends on $a(\cdot,\cdot)$, $b(\cdot,\cdot)$, $\bar{\alpha}$
as in \cref{Theor:ellipt} and $\beta$ as in \cref{Theor:infsup}. Details can be found in \cite{brezzi1991tn}.

\section{Numerical examples}\label{Sec:Results}
In this section, we use first the mixed method~\eqref{eq-fem} for two examples chosen to illustrate its ability to link regions of thermally conductive solids.  Then, and to complete the section, we show that the ideas of the mixed network-continuum formulations can be exploited to solve one particular type of inference problems in diffusive situations.

\subsection{Heat flux between dissimilar regions}
In this example, we choose a square solid with dimensions $2\times 2$ where two subregions, the first with the shape of a circle, and the second an ellipse, are linked through a conductive wire. In a Cartesian coordinate system located at the center of the square and with axes parallel to its edges, the circular region has center $(x, y) = (-0.625, -0.5625)$ and radius $0.5$. The elliptical region is centered at $(x, y)=(0.375, 0.4375)$, and has horizontal and vertical semi-axes of lengths $0.175$ and $0.3$, respectively (see \cref{fig:theoryexample}). The setting is very similar to the one we employed in \cref{sec:linkedform} to describe the mixed-dimensional boundary value problem. See \cref{Fig:sketch} for comparison.

The block has conductivity $\kappa=100$ and the wire has a high conductivity of value $\bar{\kappa}=10,000$. The temperature is constrained at the top and bottom edges to values of $500$ and $300$, respectively. If the problem had no thermal link, the thermal field would be linear with constant vertical heat flux. However, due to the presence of the linking wire, which is selected with a high conductivity, the thermal field is distorted. The elliptical region --closer to the hotter top-- serves as a heat source for the circular region --this one closer to the cooler bottom edge. As a result, the temperature fields in the two connected regions are close to 400.

In principle, the region and its center can lay anywhere in the domain. The centers of these regions determine the amount of temperature going into and coming out of the wire calculated as the average temperature in the corresponding region. To study the convergence of the formulation, we obtain six finite element solutions for this problem, with increasing resolution (see \cref{fig:theoryexample}). In each of the solutions, the linked regions $B_1, B_2$ consist of those elements of the mesh whose centers fall inside the circle and ellipse, respectively. As can be observed in \cref{fig:theoryexample}, coarse meshes do not represent accurately the linked regions (see the top two figures in \cref{fig:theoryexample}). However, when the mesh is refined, both the circle as well as the ellipse are accurately approximated (see, again, \cref{fig:theoryexample}). We note that, alternatively, we could have chosen to employ a mesh that was adapted to the linked regions, but the converged solutions would not differ.

\begin{figure}[htbp!]
    \centering
    \begin{subfigure}{0.35\textwidth}
        \centering
        \includegraphics[width=\textwidth]{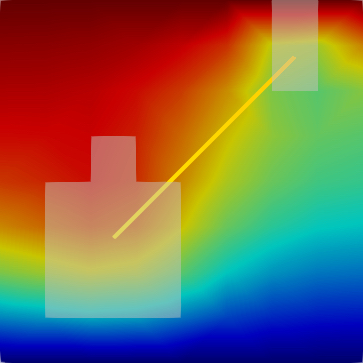}
        \caption{$8^2$ mesh.}
        \label{fig:theoryEx8}
    \end{subfigure}
    \hspace{3em}
    \begin{subfigure}{0.35\textwidth}
        \centering
        \includegraphics[width=\textwidth]{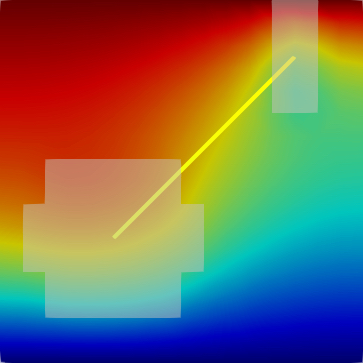}
        \caption{$16^2$ mesh.}
        \label{fig:theoryEx16}
    \end{subfigure}

    \begin{subfigure}{0.35\textwidth}
        \centering
        \includegraphics[width=\textwidth]{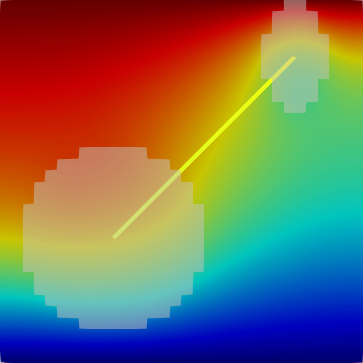}
        \caption{$32^2$ mesh.}
        \label{fig:theoryEx32}
    \end{subfigure}
    \hspace{3em}
    \begin{subfigure}{0.35\textwidth}
        \centering
        \includegraphics[width=\textwidth]{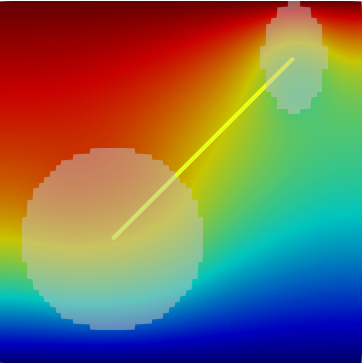}
        \caption{$64^2$ mesh.}
        \label{fig:theoryEx64}
    \end{subfigure}

    \begin{subfigure}{0.35\textwidth}
        \centering
        \includegraphics[width=\textwidth]{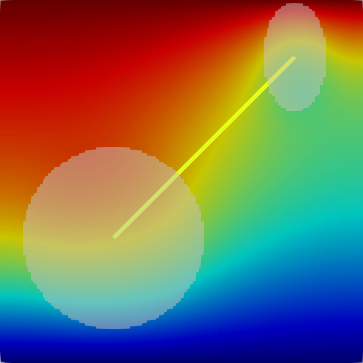}
        \caption{$128^2$ mesh.}
        \label{fig:theoryEx128}
    \end{subfigure}
    \hspace{3em}
    \begin{subfigure}{0.35\textwidth}
        \centering
        \includegraphics[width=\textwidth]{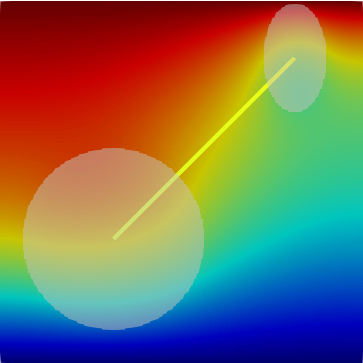}
        \caption{$256^2$ mesh.}
        \label{fig:theoryEx256}
    \end{subfigure}
\vspace*{.5em}\\
    \includegraphics[width=0.27\textwidth]{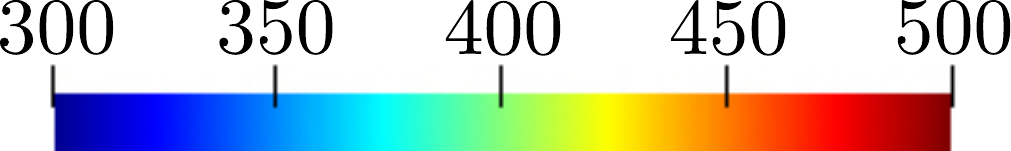}

    \caption{Thermal link connecting dissimilar regions of circle and ellipse inside a block for six meshes with $8^2$, $16^2$, $32^2$, $64^2$, $128^2$, $256^2$ elements (top left to bottom right).}
    \label{fig:theoryexample}
\end{figure}

\subsection{Convergence of increasingly finer connected regions}

This second example examines the classical Fourier problem and its relation with diffusive networks, of the type employed in this work, with the continuum notion of flux. For that, we will consider a square domain of unit side, a material  with conductivity $\kappa=100$, and temperature boundary conditions described by parabolas at the four edges, all with their maximum values at the center of the edges and zero at the vertices. The maximum temperatures at each edge are $100$ (bottom), $200$ (right), $300$ (top), and $400$ (left).

This problem can be solved with a classical numerical method, such as the finite element or finite volume method (see \cref{fig:multipleQuadsFE} for the finite element solution with a mesh of $512^2$ bilinear elements). However, here we solve the problem by discretizing the domain into disjoint and \emph{independent} square regions. These regions are connected with their neighbors through wires with conductivity $\kappa \cdot \Delta y$ (horizontal wires) and $\kappa \cdot \Delta x$ (vertical wires), where $\Delta x$ and $\Delta y$ are the lengths of the corresponding region in $x$ and $y$ direction, respectively. Then, this linked problem is solved using the formulation described in \cref{Sec:ApproxProblem}. Since we use only one finite element for each region, the method explained above is equivalent to a finite volume method where the fluxes are obtained by solving the boundary value problem of the wires.

\cref{fig:multipleQuadsAll} shows the solutions obtained with an increasing number of independent regions. In addition to the thermal field, the edges of the diffusive network are also drawn as well as their temperature field, using the same color scale as in the continuum regions. In the coarser solutions, one can clearly see the discontinuity of the thermal field across the boundaries of the regions since, as explained, they are independent meshes and do not share any node. Also, it is apparent in these coarser solutions that the thermal field in the wires differs from the continuum field at corresponding points.

Remarkably, as the size of the regions is reduced --and hence the length and conductivity of the wires-- the connected solution (\cref{fig:multipleQuadsAll}) resembles more closely the finite element solution (\cref{fig:multipleQuadsFE}) used as reference. This apparent convergence is verified by the results depicted in \cref{fig:multipleQuads_error64}. This figure shows the root-mean-square error of the temperature field of each linked solution compared with the reference finite element solution and obtained as
\begin{equation}
    e = \sqrt{\frac{1}{N} \sum_i^{N} ( \phi_i - \phi_i^{FE} )^2} \ ,
\end{equation}
where $N$ is the number of regions and $\phi_i$ refers to the temperature at the center of the $i-$th region, or element, respectively.

\begin{figure}[htbp!]
    \centering
    \begin{subfigure}{0.35\textwidth}
        \centering
        \includegraphics[width=\textwidth]{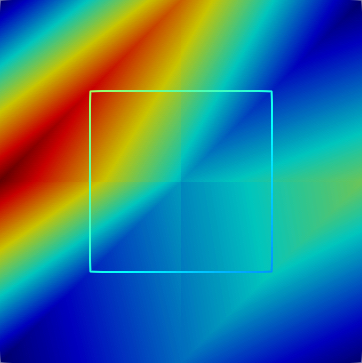}
        \caption{$2^2$ regions}
        \label{fig:multipleQuads2}
    \end{subfigure}
    \hspace{3em}
    \begin{subfigure}{0.35\textwidth}
        \centering
        \includegraphics[width=\textwidth]{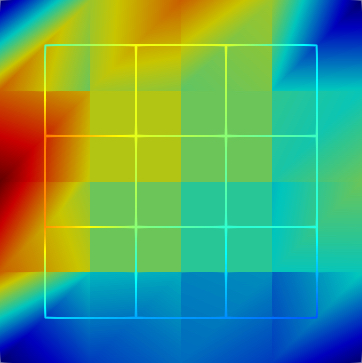}
        \caption{$4^2$ regions}
        \label{fig:multipleQuads4}
    \end{subfigure}

    \begin{subfigure}{0.35\textwidth}
        \centering
        \includegraphics[width=\textwidth]{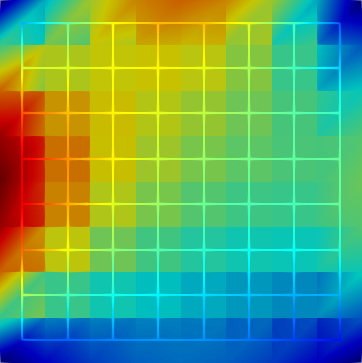}
        \caption{$8^2$ regions}
        \label{fig:multipleQuads8}
    \end{subfigure}
    \hspace{3em}
    \begin{subfigure}{0.35\textwidth}
        \centering
        \includegraphics[width=\textwidth]{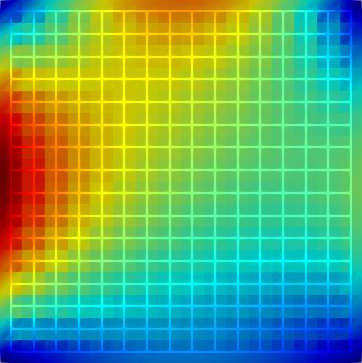}
        \caption{$16^2$ regions}
        \label{fig:multipleQuads16}
    \end{subfigure}

    \begin{subfigure}{0.35\textwidth}
        \centering
        \includegraphics[width=\textwidth]{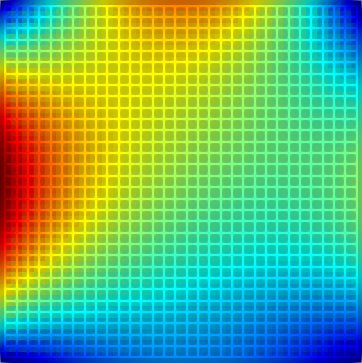}
        \caption{$32^2$ regions}
        \label{fig:multipleQuads32}
    \end{subfigure}
    \hspace{3em}
    \begin{subfigure}{0.35\textwidth}
        \centering
        \includegraphics[width=\textwidth]{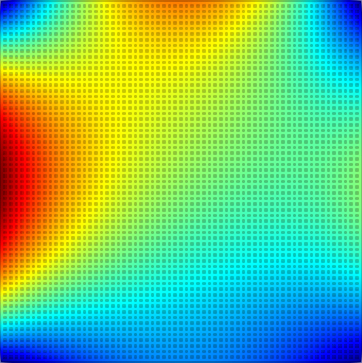}
        \caption{$64^2$ regions}
        \label{fig:multipleQuads64}
    \end{subfigure}

    \includegraphics[width=0.27\textwidth]{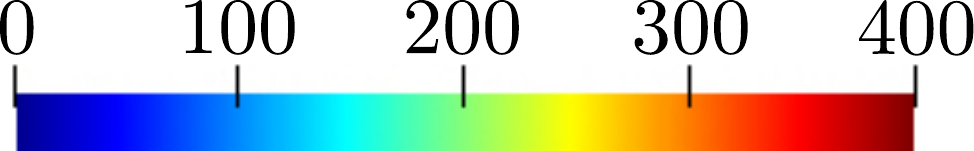}

    \caption{Temperature field in thermal problem solved with independent regions linked with wires. The linking wires are depicted as lines.}
    \label{fig:multipleQuadsAll}
\end{figure}

\begin{figure}[ht!]
	\centering
	\begin{subfigure}{0.4\textwidth}
		\centering
		\vspace{3.5mm}
		\includegraphics[width=\textwidth]{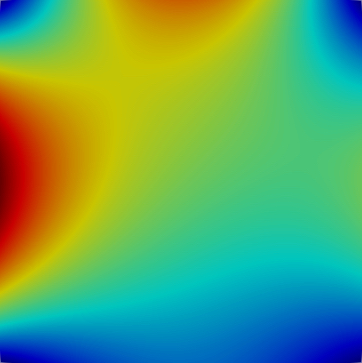}\\
		\vspace{5mm}
		\includegraphics[width=0.55\textwidth]{Figures/multipleQuadsColorbar}
		\vspace{2.75mm}
		\caption{Finite element solution of the several regions and wires problem with a quadratic mesh of $512 \times 512$ elements.}
		\label{fig:multipleQuadsFE}
	\end{subfigure}
	\hspace{1em}
	\begin{subfigure}{0.5\textwidth}
		\centering
		\includegraphics[width=\textwidth]{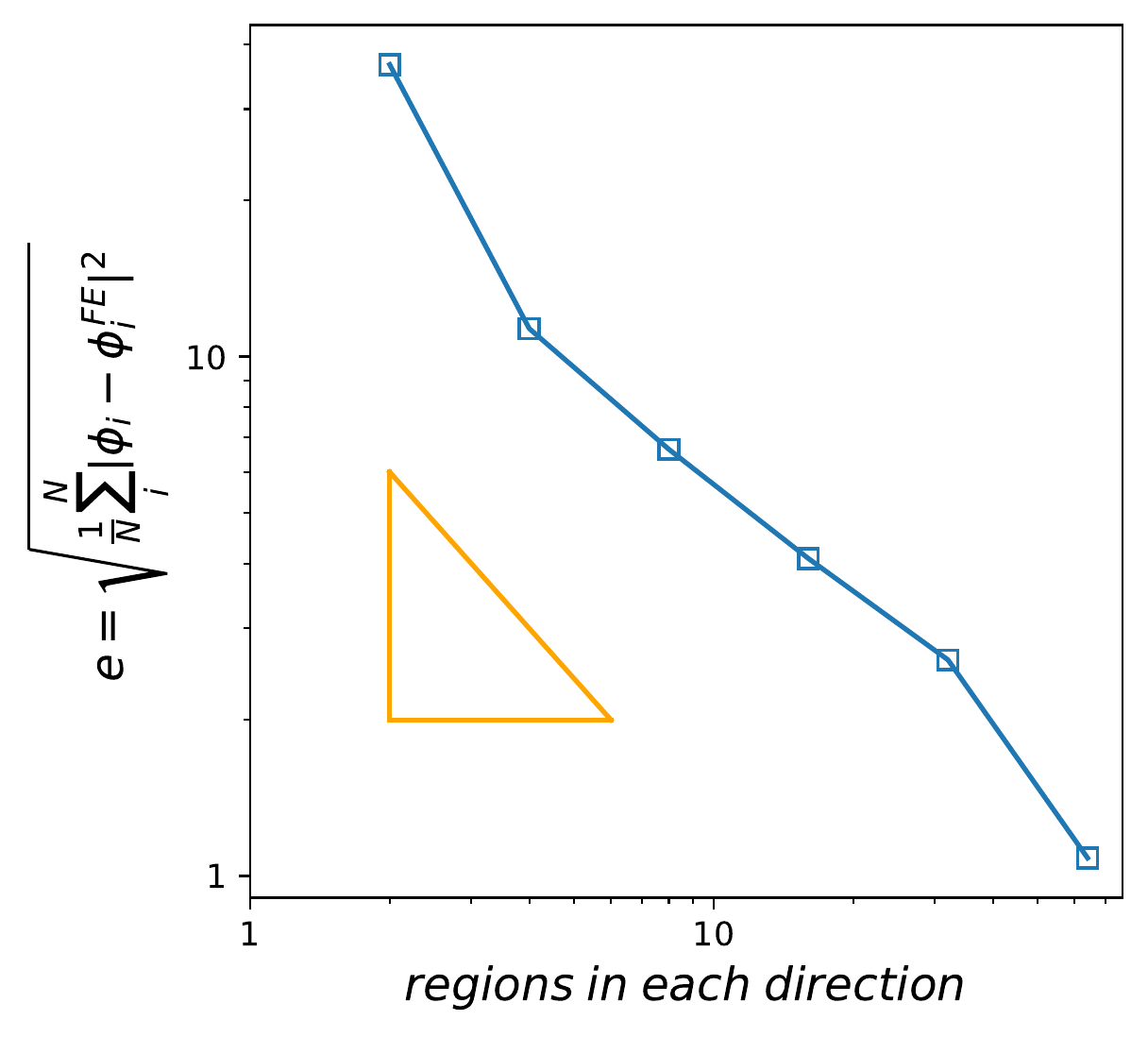}
		\caption{Root-mean-square error of the regions' center temperature compared to the finite element solution. Orange triangle represents linear convergence.}
		\label{fig:multipleQuads_error64}
	\end{subfigure}
	\caption{Finite element solution of the several regions and wires problem and root-mean-square error of the regions' center temperature compared to the finite element solution.}
	\label{fig:multipleQuads_ref_error}
\end{figure}

\subsection{Using linked models to infer diffusive solutions}
\label{subsec:partial_information}
This final example studies the possibility of using the methods introduced in \cref{Sec:ApproxProblem} to infer the optimal thermal field in a domain when only partial information is available. More precisely, we are interested in determining the most likely temperature field in a domain when there is information about the temperature on part of the boundary and the \emph{average} temperature in some regions of the domain.

When studying a Fourier-type diffusive problem, one needs to know the Dirichlet and Neumann conditions in all of the boundaries as well as the heat applied in the interior of the domain, see \cref{eq-poisson}. The lack of either of these data renders the problem ill-posed and no solution can be analytically (nor numerically) obtained. Here we are interested in finding the optimal temperature field of a continuum domain, $\Omega$, where the Dirichlet boundary conditions are partially known but also the average temperature in some regions, $\{B_i\}_{i=1}^{N_{regions}} \; \mathrm{with} \; B_i \subseteq \Omega $. In particular, the heat supplied through the Neumann boundary and the interior of the domain is completely unknown. This is an optimization problem that searches for the temperature field $\phi\in \mathcal{U}$ that verifies
\begin{equation}
    \begin{aligned}
        \min_{\phi\in\mathcal{U}} \quad &
        \mathcal{E} = \int_{\Omega} \frac{\kappa}{2} \, \| \nabla \phi \|^2 \mathrm{d}V, \\
        \mathrm{s.t.} \quad & \frac{1}{\vert B_i\vert} \int_{B_i} \phi \dV = \theta_i, \quad i \in  \{1,...,N_{regions} \} \\
        & \phi = \bar{\phi}, \quad \emptyset
        \neq \partial \Omega_j \subseteq \partial \Omega
    \end{aligned}
    \label{eq-inferenceOptimization}
\end{equation}

The solution to this optimization problem is the thermal field on the solid~$\Omega$ with imposed mean-temperature constraints using \emph{degenerated wires} that link the assumed temperature with the selected regions using the constraint given in \cref{eq-one-theta2,eq-one-theta2}.

To analyze the ability of problem~\eqref{eq-inferenceOptimization} to infer an unknown thermal field, we consider a $2 \times 1$ rectangle with $\kappa = 100$.
We compute a reference solution with boundary conditions and thermal loading depicted in \cref{fig:inferenceBVPcond}. In this figure, the temperature on the left and top edges corresponds to $\phi= 400$ and $\phi = 200$, respectively. Also, a wavy heat supply is imposed in the whole domain of the form
\begin{equation}
    h(x,y) =
    \begin{cases}
        0, & \quad  \mathrm{if} \; (x,y) \in \mathcal{C} \\
        \sin\left( \frac{4\pi (x-x_x)^2 + (y-y_c)^2}{r^2}\right), & \quad \mathrm{otherwise}
    \end{cases}
    ,
\end{equation}
where $\mathcal{C}$ is the circle with center $(x,y) = (1.5, 0.25)$ and radius $r = 0.25$, where the coordinates refer to a Cartesian system located at the bottom left corner of the rectangle, with the $x,y$ axes parallel to the horizontal and vertical directions, respectively. \cref{fig:inferenceBVPsol} shows the solution of this boundary value problem (in what follows, the reference problem) obtained with a finite element mesh of size $128 \times 64$.

For the current example, we define the potentially linked regions as the ones obtained with a simple grid of the whole domain of $8\times8$ equal rectangles. We only employ as known boundary data the temperature on the left edge. Then, we analyze a sequence of discrete problems that employ an increasingly large number of data concerning the average temperature in random regions. Denoting these regions as $B_i$ (see \cref{eq-inferenceOptimization}), we proceed as follows: first, we solve the problem with known Dirichlet data, and a known average temperature~$\theta_i$ on a random region~$B_1$ which we can obtain from the exact solution; then, we  select a second random region $B_2$, obtain its average temperature $\theta_2$ from the exact solution and solve the minimization problem with the two constraints. Proceeding in this fashion, we obtain the solutions illustrated in \cref{fig:inferenceAll}. This figure shows that when the only available data is the average temperature of one region and the Dirichlet boundary condition, the solution of minimal energy is far from the reference (see top figures in \cref{fig:inferenceAll}. As expected, the more data available, the better the obtained solution. When the average temperatures are available in all the rectangular regions, the error is minimal and the solution of the optimization problem is close to the reference solution. Naturally, the small oscillations in the solution are not captured by the approximation because the available data do not have enough resolution. One might expect that in the limit when the regions of known average become very small and cover the whole domain, the solution of the optimization problem converges to the exact solution.

Finally, in \cref{fig:inference_energyError} the energy error $\varepsilon$ --- relative to the energy of the reference solution ---, i.e.,
\begin{equation}
    \varepsilon = \frac{\vert\mathcal{E}^{N_{regions}}-\mathcal{E}^{FE}\vert}{\mathcal{E}^{FE}}
\end{equation}
for the calculated sequence is illustrated. The energy error depends on the random sequence of known data, but it should be always monotonically decreasing. Note that for this sequence, there is almost no improvement in the solution for the last $33$ added data instances. The average temperature of the region added in the $32nd$ iteration makes the difference, reducing the energy error to approximately $20\%$. From that iteration, the solutions of the constrained optimization problems are very similar. (See \cref{fig:inference49,fig:inference64}).

\begin{figure}[t]
	\centering
	\begin{subfigure}{0.45\textwidth}
		\centering
		\includegraphics[width=1.05\textwidth]{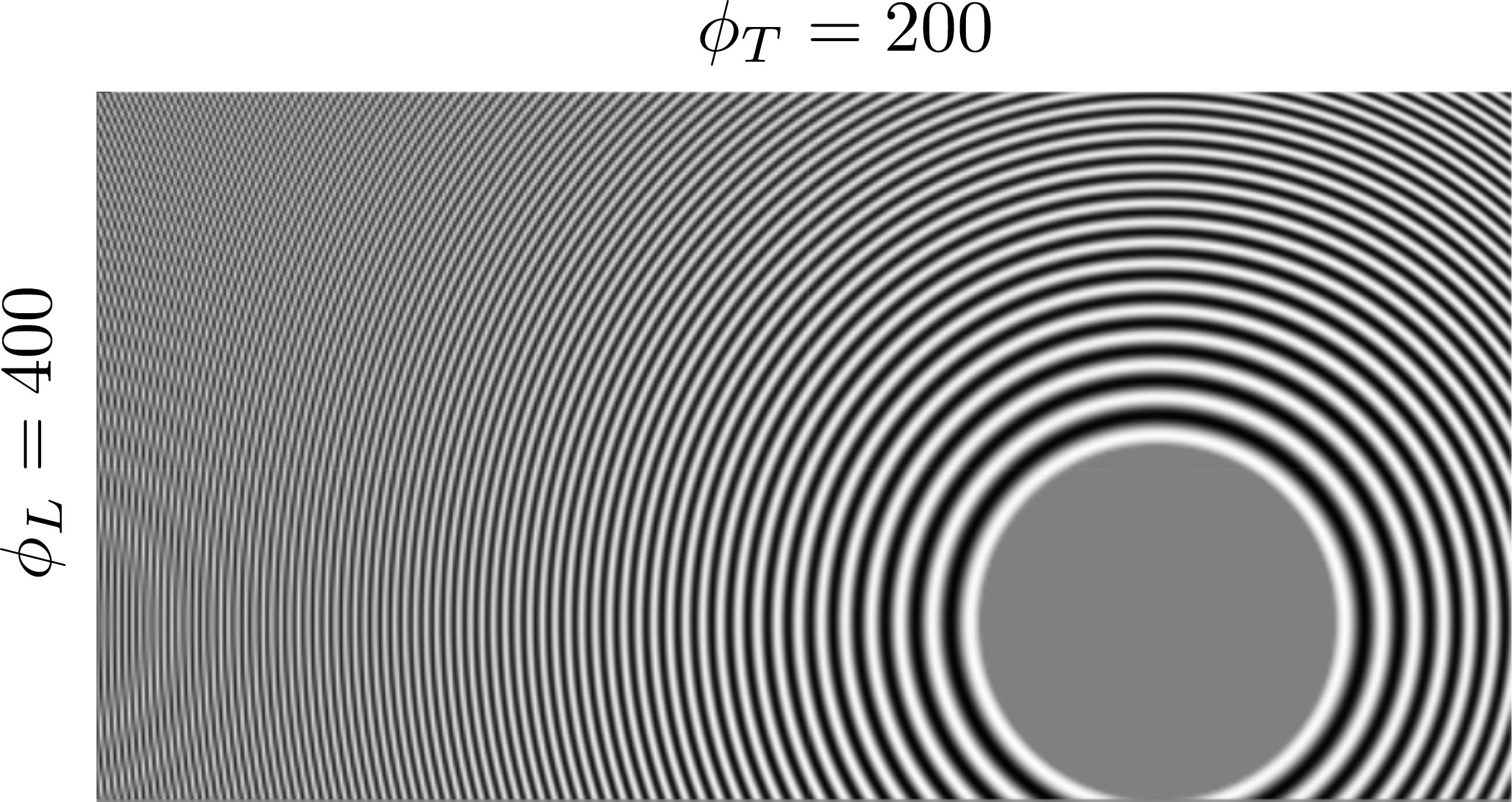}
		\vspace{8mm}
		\caption{Heat supply field and boundary conditions.}
		\label{fig:inferenceBVPcond}
	\end{subfigure}
	\hspace{1em}
	\begin{subfigure}{0.45\textwidth}
		\centering
		\vspace{6mm}
		\includegraphics[width=\textwidth]{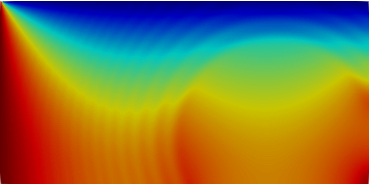} \\
		\vspace{5mm}
		\includegraphics[width=0.45\textwidth]{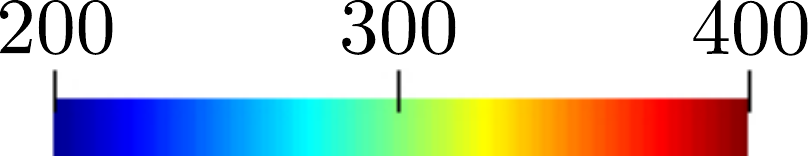}
		\caption{Reference solution with a fe mesh of size $128 \times 64$.}
		\label{fig:inferenceBVPsol}
	\end{subfigure}
	\caption{Reference boundary value problem solution of the inference example.}
	\label{fig:inferenceBVP}
\end{figure}

\begin{figure}[htbp!]
    \centering
    \begin{subfigure}{0.35\textwidth}
        \centering
        \includegraphics[width=\textwidth]{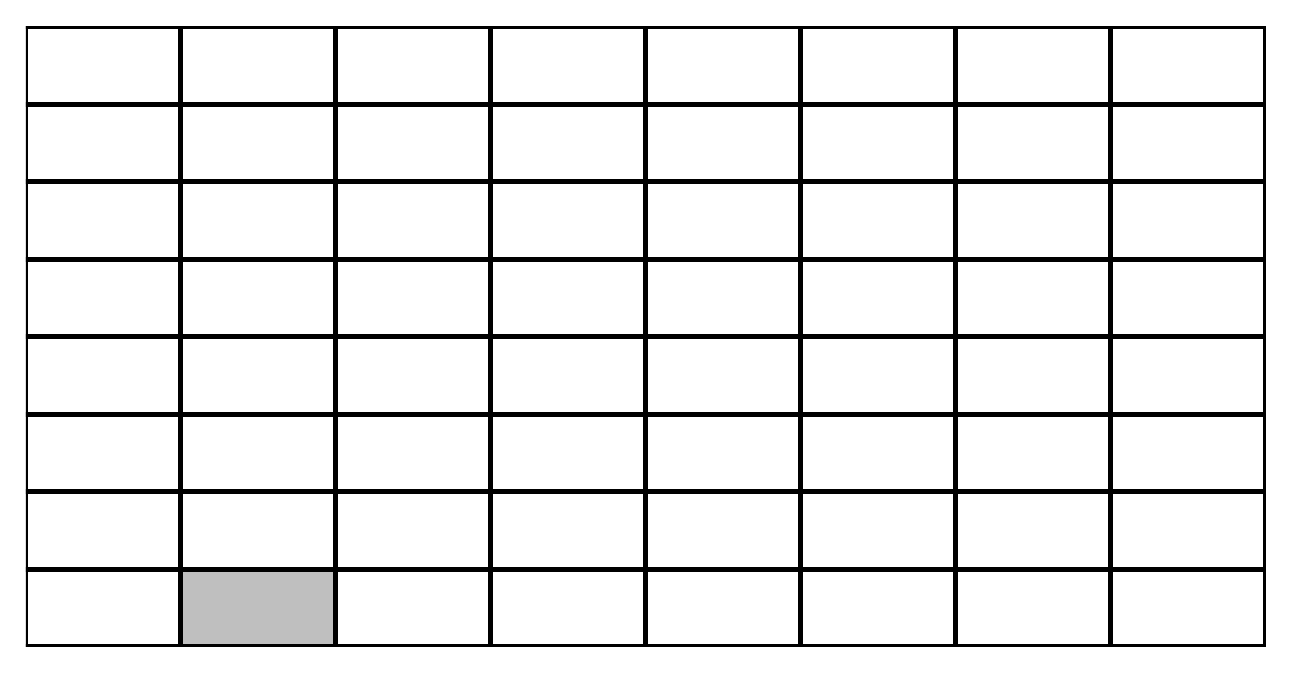}
        \caption{$1$ random region}
        \label{fig:inferenceRegions1}
    \end{subfigure}
    \hspace{3em}
    \begin{subfigure}{0.35\textwidth}
        \centering
        \includegraphics[width=\textwidth]{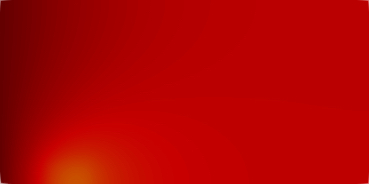}
        \caption{Solution with $1$ random region}
        \label{fig:inference1}
    \end{subfigure}

    \begin{subfigure}{0.35\textwidth}
        \centering
        \includegraphics[width=\textwidth]{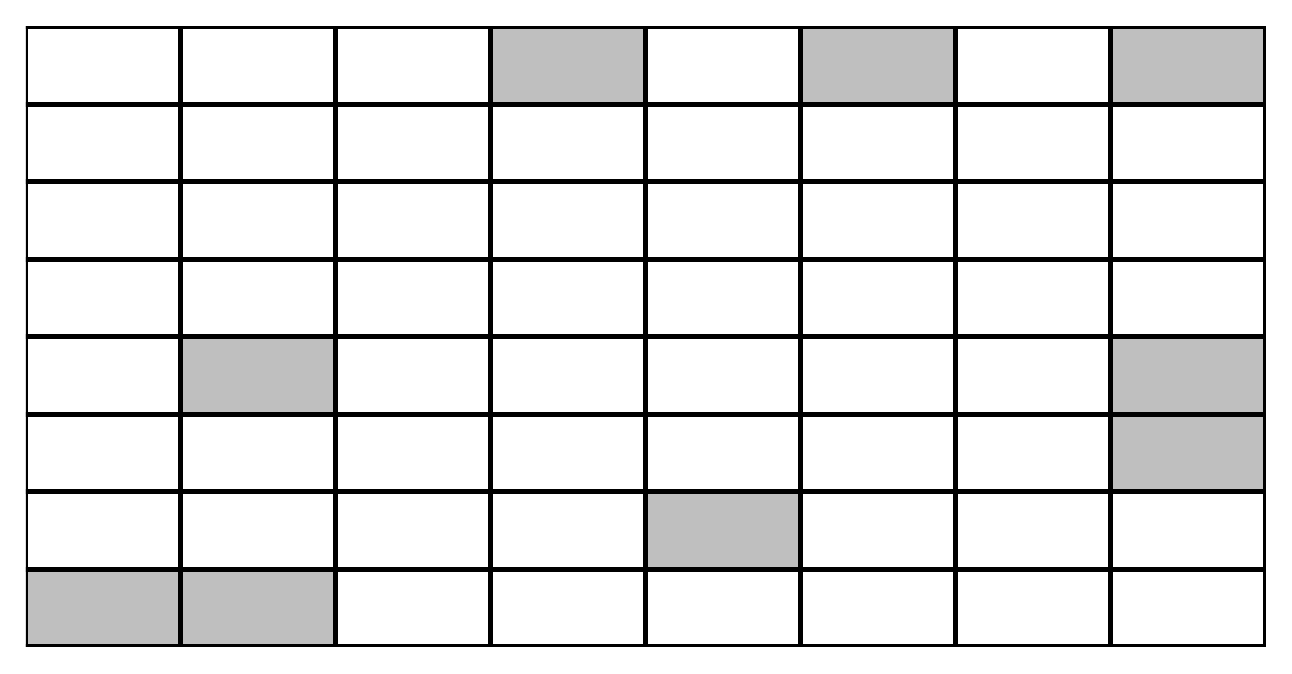}
        \caption{$9$ random regions}
        \label{fig:inferenceRegions9}
    \end{subfigure}
    \hspace{3em}
    \begin{subfigure}{0.35\textwidth}
        \centering
        \includegraphics[width=\textwidth]{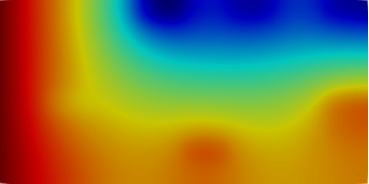}
        \caption{Solution with $9$ random regions}
        \label{fig:inference9}
    \end{subfigure}

    \begin{subfigure}{0.35\textwidth}
        \centering
        \includegraphics[width=\textwidth]{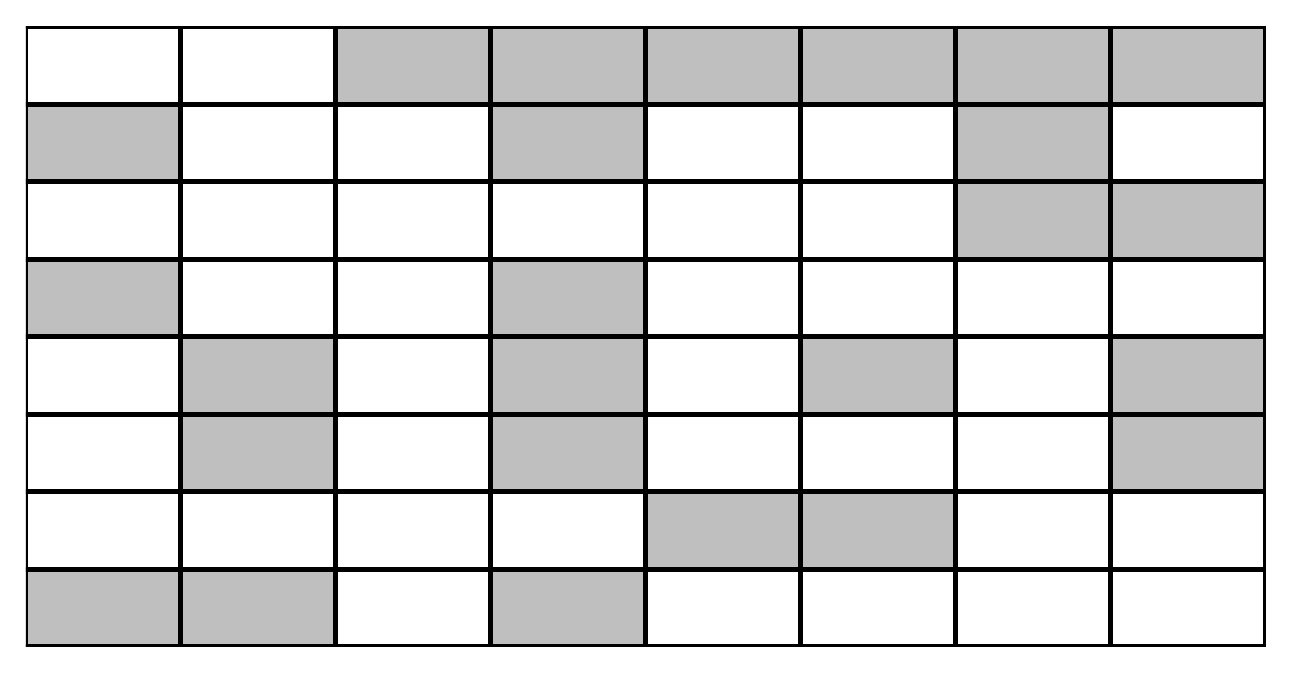}
        \caption{$25$ random region}
        \label{fig:inferenceRegions25}
    \end{subfigure}
    \hspace{3em}
    \begin{subfigure}{0.35\textwidth}
        \centering
        \includegraphics[width=\textwidth]{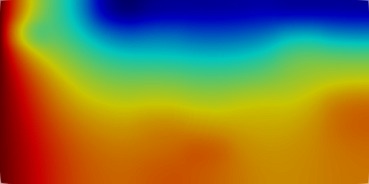}
        \caption{Solution with $25$ random regions}
        \label{fig:inference25}
    \end{subfigure}

    \begin{subfigure}{0.35\textwidth}
        \centering
        \includegraphics[width=\textwidth]{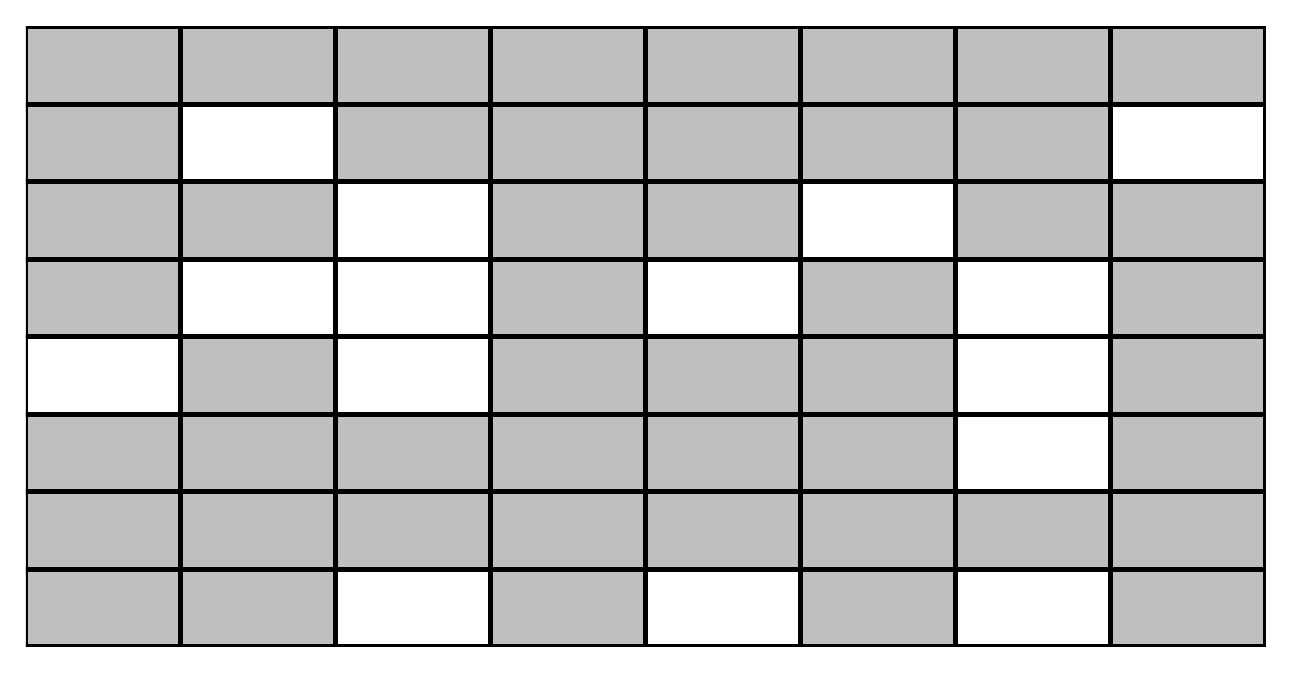}
        \caption{$49$ random regions}
        \label{fig:inferenceRegions49}
    \end{subfigure}
    \hspace{3em}
    \begin{subfigure}{0.35\textwidth}
        \centering
        \includegraphics[width=\textwidth]{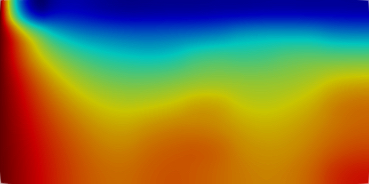}
        \caption{Solution with $49$ random regions}
        \label{fig:inference49}
    \end{subfigure}

    \begin{subfigure}{0.35\textwidth}
        \centering
        \includegraphics[width=\textwidth]{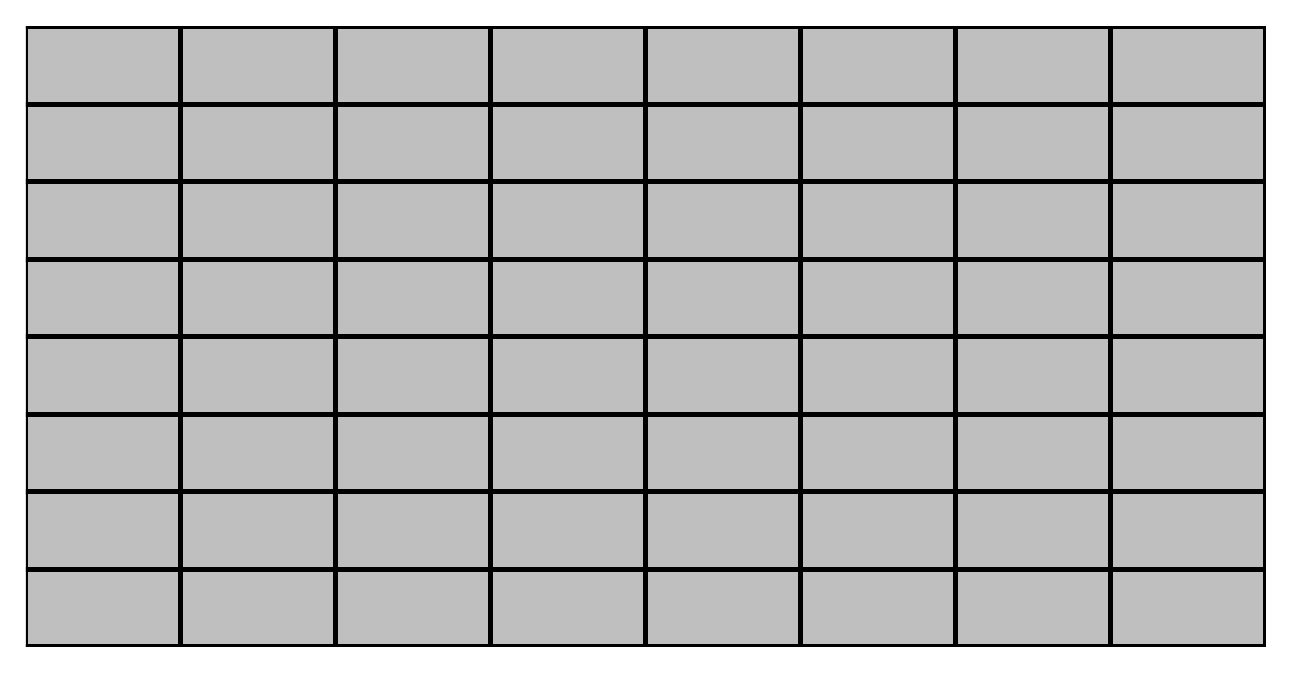}
        \caption{All regions}
        \label{fig:inferenceRegions64}
    \end{subfigure}
    \hspace{3em}
    \begin{subfigure}{0.35\textwidth}
        \centering
        \includegraphics[width=\textwidth]{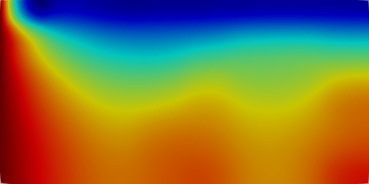}
        \caption{Solution with all regions}
        \label{fig:inference64}
    \end{subfigure}

    \begin{subfigure}{0.35\textwidth}
        \centering
    \end{subfigure}

    \begin{subfigure}{0.35\textwidth}
        \centering
        \includegraphics[width=0.65\textwidth]{Figures/inferenceColorbar}
    \end{subfigure}

    \caption{Inference problem. A sequence of solutions with a different number of
        available region data. Left: gray regions are the ones where the average temperature is available. Right: solution with the corresponding average temperatures and only left edge boundary condition known.}
    \label{fig:inferenceAll}
\end{figure}

\begin{figure}[ht!]
    \centering
    \includegraphics[width=0.5\textwidth]{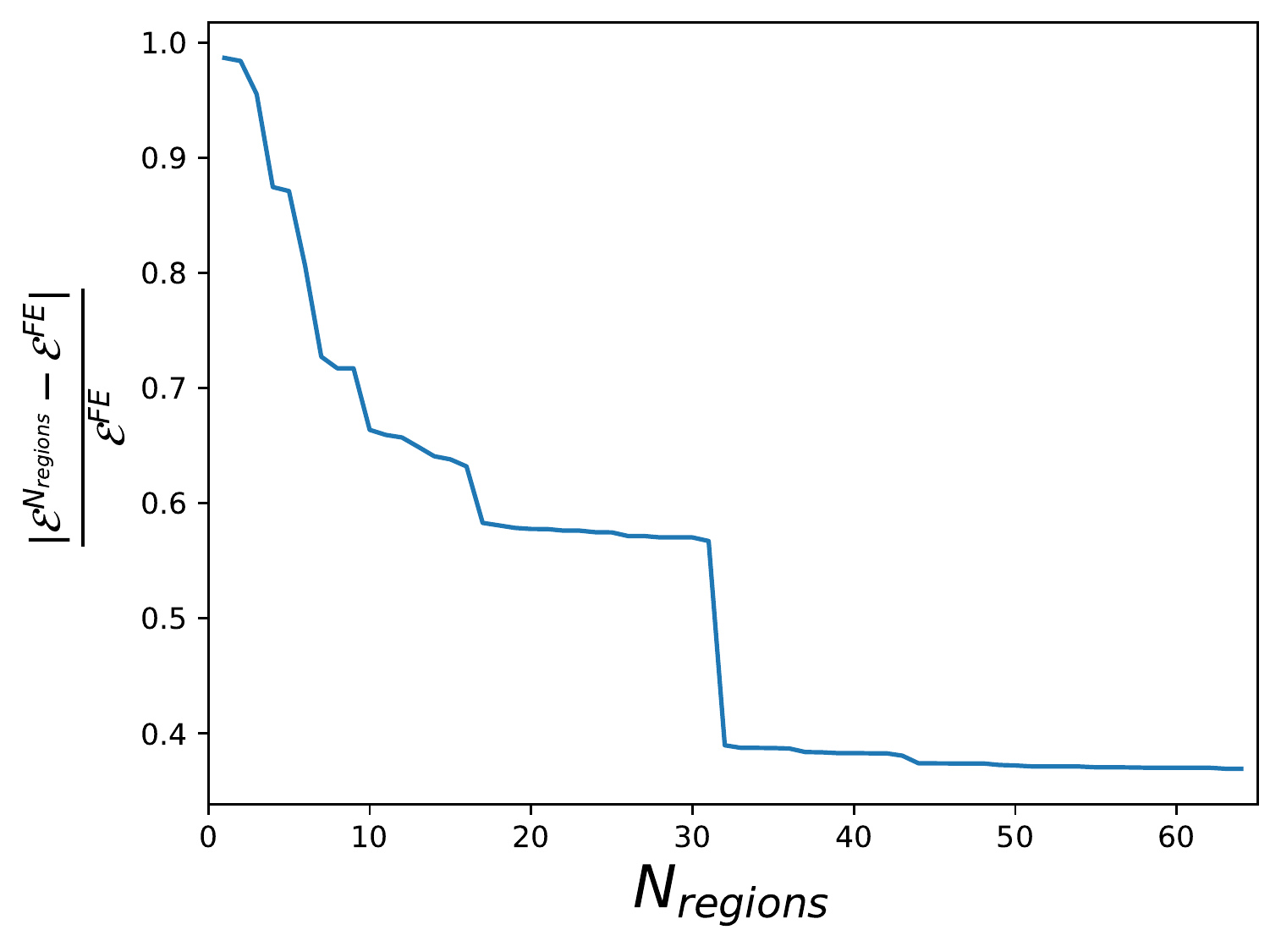}
    \caption{Energy error for the inference problem for different numbers of available region data.}
    \label{fig:inference_energyError}
\end{figure}

\section{Conclusions and main results}\label{Sec:Conclusion}
In this paper, we studied linked continuum/discrete models of diffusion from the theoretical and numerical points of view. More specifically, we consider the solution of diffusion boundary value problems where (discrete) diffusion pathways are introduced between subdomains, more precisely, between the mean values of the diffusing quantity. Although we have described all the results and examples in the language of thermal problems, the scope of the article is more general and applies to any type of linear Fourier-type boundary value problem.

After introducing the setting for the linked formulations of bulk and the one-dimensional diffusive models, we proved the well-posedness of the continuum/discrete linked problem in the corresponding functional spaces by resorting to the theory of constrained boundary value problems. Next, we proved that the well-posedness continues to hold for the approximate problem, i.e. the problem discretized via mixed finite elements. A direct consequence of the two previous results is the convergence of the fully discretized equations of the linked model to its exact solution.

The previous theoretical findings have been illustrated with three numerical examples. They show that the finite element method is always stable and that the links take care of the diffusive fluxes in a different scale than the continuum, but that in the limit both discrete and continuum mechanisms are equivalent. Finally, we showed that the ideas of the method can be used to carry out inference in the solution of diffusive problems with only partial information available about the \emph{average} solution in subsets of the whole domain.

The results of this article can be employed to link the two different classes of diffusive models, namely, the PDE-based descriptions of the continuum and diffusive network-based formulations when the network connects average values of the diffusive field computed over bounded disjoint regions.

\begin{acknowledgements}
  \myack%
\end{acknowledgements}

\bibliography{biblio}
\bibliographystyle{\mybibstyle}

\end{document}